\documentclass[11pt]{amsart}

\usepackage[left=1in,right=1in,top=1in,bottom=1in]{geometry}

\usepackage[all,cmtip]{xy}
\usepackage{amssymb}
\usepackage[english]{babel}
\usepackage{mathrsfs}

\newtheorem{theorem}{Theorem}[section]
\newtheorem{lemma}[theorem]{Lemma}

\newtheorem{proposition}[theorem]{Proposition}
\newtheorem{corollary}[theorem]{Corollary}

\theoremstyle{definition}
\newtheorem{definition}[theorem]{Definition}

\newtheorem{remark}[theorem]{Remark}

\newtheorem{question}[theorem]{Question}

\newtheorem{claim}{Claim}

\def\B{\mathcal{B}}

\def\N{\mathbb{N}}
\def\Z{\mathbb{Z}}
\def\R{\mathbb{R}}
\def\T{\mathbb{T}}
\def\V{\mathcal{V}}
\def\C{\mathfrak{c}}
\def\sa{selectively admissible}
\def\sp{strongly pseudocompact}
\def\spness{strong pseudocompactness}
\def\ssp{selectively sequentially pseudocompact}

\def\cf{\mathrm{cf}}

\begin{document}
\title[Selectively sequentially pseudocompact group topologies]{Selectively sequentially pseudocompact group topologies on torsion and torsion-free  Abelian groups} 
\author[A. Dorantes-Aldama]{Alejandro Dorantes-Aldama}
\address{Department of Mathematics, Faculty of Science, Ehime University,
  Matsuyama 790-8577, Japan}
\email{alejandro\_dorantes@ciencias.unam.mx}
\thanks{The first listed author was supported by CONACyT  of M\'exico: Estancias Posdoctorales al Extranjero propuesta No. 263464}

\author[D. Shakhmatov]{Dmitri Shakhmatov}
\address{Division of Mathematics, Physics and Earth Sciences\\
Graduate School of Science and Engineering\\
Ehime University, Matsuyama 790-8577, Japan}
\email{dmitri.shakhmatov@ehime-u.ac.jp}
\thanks{The second listed author was partially supported by the Grant-in-Aid for Scientific Research~(C) No.~26400091 by the Japan Society for the Promotion of Science (JSPS)}



\begin{abstract} 
A space $X$ is {\em \ssp} if for every family $\{U_n:n\in\N\}$ of non-empty open subsets of $X,$
one can choose a point $x_n\in U_n$ for every $n\in \N$ 
in such a way that the sequence $\{x_n:n\in\N\}$ has a convergent subsequence. Let $G$ be a group from one of the following three classes:
(i)
$\mathcal{V}$-free groups, where $\mathcal{V}$ is an arbitrary variety of Abelian groups;
(ii)
torsion Abelian groups;
(iii)
torsion-free Abelian groups.   
Under the Singular Cardinal Hypothesis SCH, we prove that if $G$ admits a pseudocompact group topology, then it can also be equipped with a \ssp\ group topology. 
Since \ssp\ spaces are strongly pseudocompact in the sense of 
Garc\'ia-Ferreira and  Ortiz-Castillo, this provides a strong positive (albeit partial) answer to a question of Garc\'ia-Ferreira and  Tomita.
\end{abstract}

\maketitle

The symbol $\N$ denotes the set of natural numbers and 
$\N^+=\N\setminus\{0\}$ denotes the set of positive integer numbers.
The symbol
$\Z$ denotes the group of integer numbers,
$\R$ the set of real numbers and $\T$ 
the circle group $\{e^{i\theta} :\theta \in \R\}\subseteq \R^2.$ 
The symbols $\omega,\omega_1$ and $\C$ stand for the first infinite cardinal, the first uncountable cardinal and the 
cardinality of the continuum, respectively.
For a set $X$, the set of all finite subsets of $X$ is denoted by 
$[X]^{<\omega}$, while $[X]^\omega$ denotes the set of all countably infinite subsets of $X$.

If $X$ is a subset of a group $G$, then $\langle X\rangle $ is the smallest subgroup of $G$ that
contains $X$. 

For groups which are not necessary Abelian we use the multiplication notation, while for Abelian groups we always use
the additive one. In particular, $e$  denotes the identity element of a group and $0$ is used for the zero element 
of an Abelian group. Recall that 
an element $g$ of a group $G$ is {\em torsion\/} if $g^n=e$
for some positive integer $n$.  
A group is {\em torsion\/} if all of its elements are torsion.
A group $G$ is {\em bounded torsion\/} if there exists $n\in \N$ such that $g^n=e$ for all $g\in G$.

Let $G$ be an Abelian group. The symbol $t(G)$ stands for the subgroup of torsion elements of $G$. 
For each $n\in \N$, note that
$nG=\{ng:g\in G\}$
 is a subgroup of $G$ and the map
$g\mapsto ng$ ($g\in G$) is a homomorphism of $G$ onto $nG$.
For a cardinal $\tau$ we denote by $G^{(\tau)}$ the direct sum of $\tau$ copies
of the group $G$. 

We will say that a sequence $\{x_n:n\in\N\}$ of points in a topological space 
$X$:
\begin{itemize}
\item
{\em converges to a point $x\in X$\/} if
 the set 
$\{n\in\N: x_n\not\in W\}$ is finite for every open neighbourhood $W$ of $x$ in $X$;
\item
is {\em convergent\/} if it converges to some point of $X$.
\end{itemize}

\section{Introduction}

In this paper, we continue the study of the class of \ssp\ spaces
introduced by the authors in \cite{DoS}.

\begin{definition}\label{definition_1} \cite{DoS}
A topological space $X$ 
is {\em \ssp\/} if for every sequence $\{U_n:n\in\N\}$ of non-empty open subsets of $X,$
one can choose a point $x_n\in U_n$ for every $n\in \N$ 
in such a way that the sequence $\{x_n:n\in\N\}$ has a convergent subsequence.
\end{definition}

A nice feature 
of the class of \ssp\ spaces 
is that it is closed under 
taking arbitrary Cartesian products:

\begin{proposition}
\label{product:ssp}
\cite[Corollary 4.4]{DoS}
Arbitrary products of \ssp\ spaces are \ssp. 
\end{proposition}

Garc\'ia-Ferreira and 
Ortiz-Castillo \cite{GO} 
have recently introduced the class of strongly pseudocompact spaces.
The original definition involves the notion of a $p$-limit for an ultrafilter $p$ on $\N$. The authors have shown in \cite[Theorem 2.1]{DoS}
that the original definition is equivalent to the following one:

\begin{definition}
\label{definition_11}
A topological space $X$ is called {\em strongly pseudocompact\/}
provided that for each sequence $\{U_n:n\in\N\}$ of  pairwise disjoint non-empty open subsets of $X$,
one can choose a point $x_n\in U_n$ for every $n\in\N$ such that
the set $\{x_n:n\in\N\}$ is not closed in $X$.
\end{definition}

In the class of topological groups, this property has been investigated by Garc\'ia-Ferreira and  Tomita in \cite{GT}.

The following implications are clear from Definitions \ref{definition_1} and  \ref{definition_11}:
\begin{equation}
\label{two:implications}
\xymatrix{\text{\ssp} \to \text{strongly pseudocompact} \to \text{pseudocompact}.}
\end{equation}

It easily follows from Definition \ref{definition_11}
that 
countably compact spaces are \sp. Since there exists a countably compact space $X$ such that $X\times X$ is not even pseudocompact, 
an analogue of Proposition \ref{product:ssp} for \sp\ spaces is false. 
Whether \spness\ is productive in the class of topological groups remains an open problem \cite{GT}.
 
Neither of the implications in \eqref{two:implications}
can be reversed even in the class of topological (Abelian) groups. An example of a pseudocompact Abelian group which is not strongly pseudocompact was constructed by Garc\'ia-Ferreira and  Tomita in \cite{GT}. 
The
authors gave a consistent example of a strongly pseudocompact Abelian group which is not \ssp\ in \cite[Example 5.7]{DoS}, 
A ZFC example of such a group is given in 
\cite{SY}, thereby answering \cite[Question 8.3(i)]{DoS}.

Garc\'ia-Ferreira and 
Tomita  asked if the last implication in \eqref{two:implications} can be reversed  when the existential quantifier is added to both sides of it.

\begin{question}
\label{question:GT}
\cite[Question 2.7]{GT}
If an Abelian group admits a pseudocompact group topology, does it admit a strongly pseudocompact group topology?
\end{question}

In view of the implications in \eqref{two:implications},
the following version of this 
question 
is even harder to answer in the positive.

\begin{question}
\label{main:problem}
If an Abelian group admits a pseudocompact group topology, does it also admit a \ssp\ group topology?
\end{question}

The goal of this paper is to provide a positive answer to 
both questions
for the following classes of groups:
\begin{itemize}
\item[(a)]
$\mathcal{V}$-free groups, where $\mathcal{V}$ is a precompact variety of groups (Corollary \ref{four:conditions:for:a:precompact:variety});
\item[(b)]
$\mathcal{V}$-free groups, where $\mathcal{V}$ is any variety of Abelian groups (Corollary \ref{corollary:for:abelian:varieties});
\item[(c)]
torsion Abelian groups (Theorem \ref{thm:for:torsion:groups});
\item[(d)]
torsion-free Abelian groups (Corollary \ref{sch_torsion-free}; see also Theorem \ref{sch_connected} for a more general result).   
\end{itemize}

Our results depend on some additional set-theoretic assumptions which we outline below.

Let $\tau$ be a cardinal.
As usual, symbol $\tau^\omega$ denotes the cardinality of the set 
$[X]^\omega$ of all countably infinite subsets of some (equivalently, every) set $X$ such that $|X|=\tau$,
while $2^\tau$ denotes the cardinality of the set of all subsets 
of some (equivalently, every) set $X$ such that $|X|=\tau$.
Symbol
$\cf(\tau)$ denotes the {\em cofinality\/} of $\tau$, i.e. the smallest cardinal $\lambda$
for which there exists a representation $\tau=\Sigma\{\tau_\alpha:\alpha<\lambda\}$ where $\tau_\alpha<\tau$ 
for all $\alpha<\lambda.$ 
Finally, $\tau^+$ denotes the smallest cardinal bigger than $\tau$.

We fix the symbol SCH for denoting the following condition on cardinals:
\begin{equation}\label{sch}
\textit{if } \tau\geq \C \textit{ is a cardinal and } \cf(\tau)\not=\omega, \textit{ then } \tau^{\omega}=\tau.
\end{equation}  

We choose the abbreviation symbol SCH for denoting the condition \eqref{sch} since it
is known to be equivalent to the Singular Cardinals Hypothesis;
see \cite[Chapter 8]{J}.
 
The Generalized Continuum Hypothesis (abbreviated GCH) states 
that $2^\tau=\tau^+$ for every infinite cardinal $\tau$.
It is well known that GCH implies SCH
and that GCH is consistent with the usual axioms ZFC of set theory.

Prikry and Silver have constructed models of set theory violating
SCH 
assuming the consistency of supercompact cardinals; see \cite[chapter 36]{J}. Devlin and Jensen proved that 
the failure of SCH implies the existence of an inner model with a large cardinal; see \cite{DJ}. 
Therefore, large cardinals are needed in order to violate SCH, and so
SCH can be viewed as a ``mild'' additional set-theoretic assumption beyond the axioms of ZFC.

Our results mentioned in items (b), (c) and (d)
above
  hold under SCH,
while we were able to prove item (a) 
only under GCH. 

It is worth mentioning that our results are close in spirit to those obtained by 
Galindo and 
Macario \cite{GM}.
Indeed, 
they prove, under SCH, that if an Abelian 
group has a pseudocompact group topology, then it 
also admits a pseudocompact group topology without infinite compact subsets. This additional property is important in the Pontryagin duality theory.
Clearly, group topologies we construct are substantially different from those in \cite{GM}, as \ssp\ spaces have many non-trivial convergent sequences \cite[Proposition 3.1]{DoS}.

We refer the interested reader to \cite{DoS} for a detailed survey of connections of selective pseudocompactness and strong pseudocompactness to other classical compactness-like notions, including the notion of sequential pseudocompactness of 
Artico, 
Marconi, 
Pelant, 
Rotter and 
Tkachenko~\cite{AMPRT}.

We finish this introduction with 
the simple proposition from \cite{DoS} which shall be used 
later:

\begin{proposition}
\label{dense:ssp}
\cite[Corollary 3.4]{DoS} If a space $X$ has a dense \ssp\ (strongly pseudocompact) subspace, then $X$ is \ssp\ (respectively, strongly pseudocompact).
\end{proposition}

\section{Building \ssp\ $\mathcal{V}$-independent sets for a variety $\mathcal{V}$ of groups}
 
Recall that a {\em variety of groups\/} is a class of (abstract) groups closed under Cartesian products, subgroups and quotients
\cite{N}.
\begin{definition}\cite[Definition 2.1]{DiS}
\label{def:V-independent}
Let $\V$ be a variety of groups and let $G$ be a group.
A subset $X$ of $G$ is said to be {\em $\V$-independent} if 
\begin{itemize}
\item[(a)] $\langle X\rangle\in\V,$ and
\item[(b)] for each map $f:X\to H\in\V$ there exists a unique homomorphism $\tilde{f}:\langle X\rangle\to H$ extending $f.$
\end{itemize}
The cardinal
	$$r_\V(G)=\sup\{|X|:X \textrm{ is a } \V \textrm{-independent subset of } G\}
	$$
is called
the {\em $\V$-rank} of $G$. 
\end{definition}

We shall need the following useful fact from \cite{DiS}.

\begin{lemma}\label{lemma2.3_DiS}
\label{lemma2.4_DiS}
Let $\V$ be a variety of groups and $X$ be a subset of a group $G$. 
\begin{itemize}
\item[(i)]
$X$   is $\V$-independent if and only if
each finite subset of $X$ is $\V$-independent. 
\item[(ii)] If $H$ is a group and
$f:G\to H$ is a homomorphism such that
$f(X)$ is a $\V$-independent subset of $H$, $\langle X\rangle\in\V$
and $f\restriction X:X\to H$ is an injection, then $f\restriction \langle X\rangle :X\to H$ is an injection
and $X$ is a $\V$-independent subset of $G$. 
\end{itemize}
\end{lemma}
\begin{proof}
Item (i) is \cite[Lemma 2.3]{DiS}. The first statement in the conclusion of item (ii) is \cite[Lemma 2.4]{DiS} and the second statement can be proved similarly to the proof of \cite[Lemma 2.4]{DiS}.
\end{proof}

The next two lemmas are the main technical tool in this  paper. 

\begin{lemma}\label{indep_seq_pseudo}
Let $\V$ be a variety of groups. Suppose that $H\in\V$ is a compact metric group with $r_\V(H)\geq\omega$. Let $\sigma$ be
an
infinite cardinal such that  $\sigma ^{\omega }=\sigma $.
Then there exist a dense $\mathcal V$-independent \ssp\  subset $Y$ of $H^{\sigma}$ and a set $C\subseteq \sigma$
such that $|Y|= |C|=|\sigma\setminus C|=\sigma$ and 
$\langle Y\rangle\cap H^{C}=\{e\}$, where 
\begin{equation}
\label{def:H^C}
H^{C}=\{f\in H^{\sigma}:f(\gamma)=e\mbox{ for all }\gamma\in \sigma\setminus C\}.
\end{equation}
\end{lemma}
\proof
A subset $B$ of $H^{\sigma}$ shall be called {\em basic\/}
if $B=\prod_{\gamma\in\sigma} B_\gamma$, 
where $B_\gamma\subseteq H$ for all $\gamma\in\sigma$ and the set
$\mathrm{supp}(B)=\{\gamma\in\sigma: B_\gamma\not= H\}$ is finite.

Let $\B$ be a base for $H^{\sigma}$ of size $\sigma $ consisting of  basic open sets.
Since $|\mathcal{B}^\omega|=|\mathcal{B}|^\omega=\sigma^\omega=\sigma$,
we can enumerate 
\begin{equation}
\label{eq:U}
\mathcal{B}^\omega
=\{\{U_{\alpha ,n}:n\in\N \}: \alpha \in \sigma \}.
\end{equation}

Let $\alpha \in\sigma$. 
Since supp$(U_{\alpha ,n})$
is finite for every $n\in \N,$ 
the set 
$S_\alpha =\bigcup _{n\in\N }\mathrm{supp}(U_{\alpha ,n})$ 
is a countable subset of $\sigma.$
For every $n\in\N $ take a point $x_{\alpha ,n}\in U_{\alpha ,n}.$ 
Since $H^{S_\alpha }$ is compact metric,  
there are 
an infinite set $J_\alpha  \subseteq \N $ and a point 
$x_{\alpha,\omega }\in H^\sigma$ 
such that
  $\{x_{\alpha,n}\restriction_{S_\alpha}:n\in J_\alpha\}$ converges to 
  $x_{\alpha,\omega }\restriction_{S_\alpha}$.

Since $\sigma$ is infinite, there exists a 
subset $C$ of $\sigma$ such that $|C|=|D|=\sigma$, 
where $D=\sigma\setminus C.$ 
Since the set  $[\sigma\times(\omega+1)]^{<\omega}$ of all finite subsets of $\sigma\times(\omega+1)$
has cardinality $\sigma=|D|$,
we can fix an enumeration 
$$[\sigma\times(\omega+1)]^{<\omega}
=\{F_\gamma :\gamma\in D\}
$$
such that
\begin{equation}
\label{cofinitely:many}
|\{\gamma\in D: F_\gamma=F\}|=\sigma
\mbox{ for every } 
F\in [\sigma\times(\omega+1)]^{<\omega }.
\end{equation}

Since $r_\V(H)\ge\omega$, 
for every $\gamma\in\sigma$ we can fix an injection $h_\gamma: F_\gamma\to H$ such that $h_\gamma(F_\gamma)$ is a $\V$-independent subset of $H$. 

For every $(\alpha,n)\in \sigma\times(\omega+1)$,
define $y_{\alpha,n}\in H^\sigma$ by
\begin{equation}
\label{eq:y}
y_{\alpha,n}(\gamma)=
\left\{\begin{array}{ll}
x_{\alpha,n}(\gamma) & \mbox{ if } \gamma\in S_\alpha;\\
h_\gamma(\alpha,n) & \mbox{ if }\gamma\in D\setminus S_\alpha  \mbox{ and }
(\alpha,n)\in F_\gamma;\\
h_{\gamma}(\alpha,\omega) & \text{ if }\gamma\in D\setminus S_\alpha, (\alpha,n)\not\in F_\gamma  \mbox{ and }
(\alpha,\omega)\in F_\gamma;\\
e & \mbox{ otherwise} 
\end{array}
\right.
\end{equation}
for all $\gamma\in\sigma$.
Finally, let
\begin{equation}
\label{eq:Y}
Y=\{y_{\alpha ,n}:(\alpha,n)\in\sigma\times(\omega+1) \}.
\end{equation}

\begin{claim}
\label{claim:1a}
$y_{\alpha,n}\in U_{\alpha,n}$ for every $(\alpha,n)\in \sigma\times \omega$.
\end{claim}
\begin{proof}
Since $\mathrm{supp}(U_{\alpha,n})\subseteq S_\alpha$,
$x_{\alpha,n}\in U_{\alpha,n}$ and 
$y_{\alpha,n}\restriction_{S_\alpha}=x_{\alpha,n}\restriction_{S_\alpha}$
by \eqref{eq:y},
we get $y_{\alpha,n}\in U_{\alpha,n}$.
\end{proof}

\begin{claim}
\label{claim:2}
$Y$ is $\V$-independent and $\langle Y\rangle\cap H^{C}=\{e\}$.
\end{claim}
\begin{proof}
By 
Lemma \ref{lemma2.3_DiS}~(i), 
in order to prove that $Y$ is $\V$-independent, 
it suffices to show that every finite subset $X$ of $Y$
is $\V$-independent.
Similarly, in order to show that  $\langle Y\rangle\cap H^{C}=\{e\}$, it suffices to check that
$\langle X\rangle\cap H^{C}=\{e\}$ for every finite subset $X$ of $Y$.
Therefore, 
we fix an arbitrary finite subset $X$ of $Y$, and we are going to
prove that $X$ is $\V$-independent and satisfies $\langle X\rangle\cap H^{C}=\{e\}$.

Since $X$ is a finite subset of $Y$, we can use
\eqref{eq:Y} to fix a
finite set $F\subseteq  \sigma\times(\omega+1)$
such that
\begin{equation}
\label{eq:X}
X=\{y_{\alpha,n}:(\alpha,n)\in F\}.
\end{equation}
Since $F$ is finite, so is the set $A=\{\alpha\in\sigma: (\alpha,n)\in F$ for some $n\in\omega+1\}$. Therefore,
the set $S=\bigcup_{\alpha\in A} S_\alpha$ is at most countable.
Since $\sigma^\omega=\sigma\ge\omega$, the cardinal $\sigma$ is uncountable. Since $|\{\gamma\in D: F_\gamma=F\}|=\sigma$ by
\eqref{cofinitely:many}, there exists 
$\gamma \in D \setminus S$ such that $F_\gamma=F$.
Let $\pi_\gamma:H^\sigma\to H$ be the projection on the $\gamma$'s coordinate.

To prove that $X$ is $\V$-independent, it suffices to check 
that $G=H^\sigma$, $X$ and $f=\pi_\gamma$ satisfy the assumptions of  Lemma \ref{lemma2.3_DiS}~(ii).

Since $\V$ is a variety of groups, it is closed under taking arbitrary products and subgroups.
Since
$H\in\V$ and $\langle X\rangle$ is a subgroup of $H^\sigma$, this implies $\langle X\rangle\in\V$.
 
Let $g:F\to X$ be the map defined by $g(\alpha,n)=y_{\alpha,n}$ for $(\alpha,n)\in F$.
Suppose that $(\alpha,n)\in F=F_\gamma$. Since $S_\alpha\subseteq S$ and $\gamma\in D\setminus S$, we get $\gamma\in D\setminus S_\alpha$, so
$$
\pi_\gamma(g(\alpha,n))=\pi_\gamma(y_{\alpha,n})=y_{\alpha,n}(\gamma)=h_\gamma(\alpha,n)
$$
 by \eqref{eq:y}. 
This shows that $\pi_\gamma\circ g=h_\gamma$.
Since $h_\gamma$ is an  injection, so is $\pi_\gamma\circ g=\pi_\gamma\restriction_X\circ g$. 
Since $g$ is a surjection, this implies that $\pi_\gamma\restriction_X$ is an injection.
Finally, $\pi_\gamma(X)=\pi_\gamma(g(F))=h_\gamma(F)$ is a $\V$-independent subset of $H$
by the choice of $h_\gamma$. 

Applying Lemma \ref{lemma2.3_DiS}~(ii), we conclude that 
$X$ is $\V$-independent and 
$\pi_\gamma\restriction_{\langle X\rangle}$ is an injection.

 Let $x\in \langle X\rangle\setminus \{e\}$ be arbitrary.
Then $x(\gamma)=\pi_\gamma(x)\not=e,$ because $ \pi_\gamma$
is injective on $\langle X\rangle.$ 
 Since $\gamma\in D=\sigma\setminus C$, it follows from \eqref{def:H^C} that  $x\not\in H^C$.
 This proves the equation
$\langle X\rangle\cap H^{C}=\{e\}$.
\end{proof} 

\begin{claim}
$Y$ is \ssp.
\end{claim}
\begin{proof}
Let $\{U_n:n\in\N\}\in
\B^\omega$.
By \eqref{eq:U}, there exists $\alpha\in\sigma$ such that
$U_n=U_{\alpha,n}$ for every $n\in\N$.
Since $y_{\alpha,n}\in U_{\alpha,n}=U_n$ for every $n\in\N$
by Claim \ref{claim:1a},
it suffices to prove that the subsequence
$\{y_{\alpha ,n}:n\in J_\alpha  \}$ of the sequence $\{y_{\alpha,n}:n\in\N\}\subseteq Y$
converges to $y_{\alpha ,\omega }\in Y$. 
Note that  the sequence $\{y_{\alpha,n}\restriction_{S_\alpha}:n\in J_\alpha\}=\{x_{\alpha,n}\restriction_{S_\alpha}:n\in J_\alpha\}$ converges to 
  $x_{\alpha,\omega }\restriction_{S_\alpha}=y_{\alpha,\omega}\restriction_{S_\alpha}$
by \eqref{eq:y} and the choice of $\{x_{\alpha,n}:n\in\N\}$ and $J_\alpha$.
Therefore, it suffices to show that 
the sequence $\{y_{\alpha,n}(\gamma):n\in \N\}$ converges to 
$y_{\alpha,\omega}(\gamma)$ for every $\gamma\in \sigma\setminus S_\alpha$.
We consider two cases.

\smallskip
{\em Case 1\/}.
$\gamma\in D\setminus S_\alpha$.
If
$(\alpha,\omega)\in F_\gamma$,
then \eqref{eq:y} implies that $y_{\alpha,\omega}(\gamma)=h_\gamma(\alpha,\omega)$
and $\{n\in\N: y_{\alpha,n}(\gamma)\not=h_\gamma(\alpha,\omega)\}\subseteq
\{n\in\N:(\alpha,n)\in F_\gamma\}$.
Since $F_\gamma$ is finite, the sequence 
$\{y_{\alpha,n}(\gamma):n\in \N\}$ converges to 
$y_{\alpha,\omega}(\gamma)$.
Suppose now that $(\alpha,\omega)\not\in F_\gamma$.
Then $y_{\alpha,\omega}(\gamma)=e$ 
and $\{n\in\N: y_{\alpha,n}(\gamma)\not=e\}\subseteq
\{n\in\N:(\alpha,n)\in F_\gamma\}$
by \eqref{eq:y}.
Since $F_\gamma$ is finite, the sequence 
$\{y_{\alpha,n}(\gamma):n\in \N\}$ converges to 
$y_{\alpha,\omega}(\gamma)$.

\smallskip
{\em Case 2\/}.
$\gamma\in \sigma\setminus (D\cup S_\alpha)$.
In this case \eqref{eq:y} implies that $y_{\alpha,n}(\gamma)=e$ for all $n\in\omega+1$.
Therefore,  the sequence 
$\{y_{\alpha,n}(\gamma):n\in \N\}$ converges to 
$y_{\alpha,\omega}(\gamma)$.
\end{proof}

\begin{claim}
$Y$ is dense in $H^\sigma$.
\end{claim}
\begin{proof}
Let $V$ be a non-empty open subset of $H^\sigma$.
We need to show that $Y\cap V\not=\emptyset$.
By taking a smaller set $V$ if necessary, we may assume that
$V$ is a basic open subset of $H^\sigma$.
Let $U_n=V$ for every $n\in\N$.
Then $\{U_n:n\in\N\}\in\B^\omega$,
so by \eqref{eq:U}, there exists $\alpha\in\sigma$ such that 
$U_n=U_{\alpha,n}$ for every $n\in\N$.
Since $y_{\alpha,1}\in U_{\alpha,1}=V$ by Claim \ref{claim:1a}
and 
$y_{\alpha,1}\in Y$ by \eqref{eq:Y}, we get  $Y\cap V\not=\emptyset$.
\end{proof}

It follows from \eqref{eq:Y} and Claim \ref{claim:2} that
$|Y|=|\sigma\times(\omega+1)|=\sigma+\omega=\sigma$.
\endproof

\begin{lemma}\label{v-free-group}
Let $\sigma$ and $\tau$ be infinite cardinals satisfying $\sigma^{\omega} =\sigma\leq\tau\leq 2^{\sigma}$.
If $\V$ is a variety of Abelian groups and $H\in\V$ is a compact metric group with 
$r_\V(H)\geq\omega$, then 
there exist a dense $\mathcal V$-independent \ssp\  subset $X$ of $H^{\sigma}$ and a set $D\subseteq \sigma$
such that $|X|=\tau, |D|=|\sigma\setminus D|=\sigma$ and 
$\langle X\rangle\cap H^{D}=\{0\}$, where 
\begin{equation}
\label{eq:H^D}
H^{D}=\{f\in H^{\sigma}:f(\gamma)=0\mbox{ for all }\gamma\in \sigma\setminus D\}.
\end{equation}
\end{lemma}
\begin{proof}
By Lemma \ref{indep_seq_pseudo}, 
there exist a dense $\mathcal V$-independent \ssp\ subset $Y$ of $H^{\sigma}$ and a set $C\subseteq \sigma$
such that $|Y|=|C|=|\sigma\setminus C|=\sigma$ and 
$\langle Y\rangle\cap H^{C}=\{0\}$. 

By \cite[Lemmas 3.4 and 4.3]{DiS},
there exist a $\mathcal V$-independent subset $Z$ of $H^{C}$ and a set $D\subseteq C$
such that $|Z|=\tau, |D|=|C\setminus D|=\sigma$ and 
$\langle Z\rangle\cap H^{D}=\{0\}$. Since $Z\subseteq H^{C}$ and $H^{C}$ is a subgroup of $H^{\sigma}$, we have
$\langle Z\rangle\subseteq H^{C}$.
Since $\langle Y\rangle\cap H^{C}=\{0\}$, this implies
$\langle Y\rangle\cap \langle Z\rangle=\{0\}.$

We claim that $X=Y\cup Z$ and $D$ are as required. Indeed,
$|X|=|Y|+|Z|=\sigma+\tau=\tau$.
Since $Y$ is dense in $H^\sigma$ and $Y\subseteq X$, $X$ is dense in $H^\sigma$ as well.
Since $Y$ is dense in $H^\sigma$ and $Y\subseteq X$, $Y$ is dense in $X$. Since $Y$ is \ssp, Proposition \ref{dense:ssp} implies that $X$ is \ssp\ as well.

Next, we shall prove that $\langle X\rangle\cap H^D=\{0\}$.
Since $\langle Z\rangle\cap H^{D}=\{0\}$, 
it suffices to check that $\langle X\rangle\cap H^D\subseteq \langle Z\rangle$.
Let $x\in \langle X\rangle\cap H^D$ be arbitrary. 
We are going to show that $x\in \langle Z\rangle$.
Since $x\in\langle X\rangle$, $X=Y\cup Z\subseteq H^\sigma$ and the latter group is Abelian, 
we can find 
elements $y\in \langle Y\rangle$
and $z\in \langle Z\rangle$ such that $x=y+z.$  

Consider an arbitrary $\gamma\in \sigma\setminus C$.
Since $D\subseteq C$, we have $\gamma\in \sigma\setminus D$,
and so 
$ x(\gamma)=0$ by $x\in H^D$ and \eqref{eq:H^D}.
Since $z\in H^D$ and $\langle Z\rangle\cap H^{D}=\{0\}$,
we also get $z(\gamma)=0$.
Therefore,
$0= x(\gamma)=y(\gamma)+z(\gamma)=y(\gamma)$.
Since this holds for all  $\gamma\in \sigma\setminus C$,
from \eqref{def:H^C} we conclude that $y\in H^C$.
Since $y\in\langle Y\rangle$, we get 
$y\in \langle Y\rangle\cap H^{C}=\{0\}$,
which implies $y=0.$ Therefore, $x=z\in  \langle Z\rangle.$  

Finally, we are going to check that 
the set $X=Z\cup Y$ is $\V$-independent.
To
show this, we need to check items (a) and (b) of  Definition \ref{def:V-independent}.

(a) Since
$\langle X\rangle$ is a subgroup of $H^\sigma$, $H\in\V$ and $\V$ is a variety of groups, $\langle X\rangle \in\V$. 

(b)
Let $f:X\to A\in\V$ be an arbitrary map.
Since $Y$ is $\V$-independent,
there exists a unique homomorphism $g:\langle Y\rangle\to A$ such that
$g\restriction_Y=f\restriction_Y$.
Similarly, since 
$Z$ is $\V$-independent,
there exists a unique homomorphism $h:\langle Z\rangle\to A$ such that
$h\restriction_Z=f\restriction_Z$.
Since $\langle Y\rangle\cap\langle Z\rangle =\{0\}$ and $H^\sigma$ is an Abelian group, 
$\langle X\rangle=\langle Y\rangle\oplus \langle Z\rangle$, so 
there exists a unique homomorphism
$\tilde{f}: \langle X\rangle \to A$ extending both $g$ and $h$.
Clearly, $\tilde{f}\restriction_X=f$.
\end{proof}

\section{Selectively sequentially pseudocompact topologies on $\V$-free groups}

\begin{definition}\cite{DiS}
\label{def:admissible:cardinal}
An infinite cardinal $\tau$ is called {\em admissible\/} if there exists a pseudocompact group of 
cardinality $\tau$.
\end{definition}

\begin{definition}\label{def_sa}
We shall say that a cardinal $\tau$ is {\em \sa} if either $\tau$ is finite or there exists an infinite cardinal $\sigma$ such that
$\sigma^{\omega}=\sigma\leq\tau\leq 2^{\sigma}.$
\end{definition}

We shall need the 
following lemma describing properties of (selectively) admissible cardinals.

\begin{lemma}\label{3.4DiS}\cite[Lemma 3.4]{DiS}
\begin{itemize}
\item[(i)] Infinite selectively admissible cardinals are admissible.
\item[(ii)] Under SCH, every admissible cardinal is selectively admissible.
\item[(iii)] Under GCH, $\sigma=\sigma^{\omega}$ holds for every admissible cardinal $\sigma$.
\end{itemize}
\end{lemma}

\begin{definition}\cite[Definitions 2.1 and 2.2]{DiS}
\label{def:V-free}
Let $\V$ be a variety of groups.
\begin{enumerate}
\item[(i)] A subset $X$ of a group $G$ is a {\em $\V$-base} of $G$ if $X$ is $\V$-independent and  
$\langle X\rangle=G.$
\item[(ii)] A group is {\em $\V$-free} if it contains a $\V$-base.   
\item[(iii)] 
For every cardinal $\tau $, we denote by $F_\tau (\V)$ the unique (up to isomorphism) $\V$-free group having a $\V$-base of 
cardinality $\tau .$ 
\end{enumerate}
\end{definition}

\begin{lemma}
\label{lem:2.5}
Let $\V$ be a variety of groups.
Assume that $H$ 
 is a compact group, $\sigma, \tau$ are infinite cardinals and $X$ is a 
dense $\V$-independent \ssp\ subset of $H^\sigma$ such that 
$|X|=\tau$.
Then $G=\langle X\rangle\cong F_\tau(\V)$ is a dense \ssp\ $\V$-free subgroup of $H^\sigma$. Moreover, 
\begin{itemize}
\item[(i)] if $H$ is connected, then so is $G$;
\item[(ii)] if $H$ is locally connected, then so is $G$;
\item[(iii)] if $H$ is zero-dimensional, then so is $G$. 
\end{itemize}  
\end{lemma} 
\begin{proof}
The isomorphism $\langle X\rangle\cong F_\tau(\V)$ is clear.
Since $X\subseteq \langle X\rangle$ and $X$ is dense in $H^\sigma$, so is 
$G=\langle X\rangle $.
Since $X$ is dense in $H^\sigma$, it is also dense in $G$.
Since $X$ is \ssp, Proposition \ref{dense:ssp} implies that $G$ is \ssp\ as well. Note that the completion of $G$ coincides with $H^\sigma$,
so the rest of the statements follows from \cite[Fact 2.10]{DiS}.
\end{proof}

\begin{definition}\label{def_variety}\cite[Definition 5.2]{DiS}
A variety $\V$ is \emph{precompact} if there exists a compact zero-dimensional metric group $H\in\V$ with $r_\V(H)\geq \omega .$  
\end{definition}

Non-precompact varieties are not easy to come by, as many of the known varieties are precompact.
Indeed, any variety consisting of Abelian groups,  the variety of all groups, the variety of all nilpotent groups, the variety of all 
polynilpotent groups, the variety of all soluble groups  are known to be precompact \cite{N}.
For every prime number $p>665$, the Burnside variety $\B_p$ consisting of all groups satisfying the identity $x^p=e$
is not precompact \cite{D}.  

\begin{theorem}
\label{precompact:variety:theorem}
Let $\V$ be a precompact variety of groups. 
Then for every infinite cardinal $\sigma$ such that $\sigma^{\omega}=\sigma$, the group $F_\sigma(\V)$ admits a zero-dimensional \ssp\ group topology.
\end{theorem}
\begin{proof}
Since $\V$ is a precompact variety, by Definition \ref{def_variety}, there is a compact zero-dimensional metric 
group $H\in\V$ with $r_\V(H)\geq\omega.$ 
By Lemma 
\ref{indep_seq_pseudo}, 
there exists a dense $\mathcal V$-independent \ssp\ subset $Y$ of $H^{\sigma}$ 
such that $|Y|=\sigma.$
Applying Lemma \ref{lem:2.5} (with $X=Y$ and $\tau=\sigma$), we conclude that 
$\langle Y\rangle\cong F_\sigma(\V)$
is a zero-dimensional \ssp\ (dense) subgroup of $H^\sigma$.
The subgroup topology that $\langle Y\rangle\cong F_\sigma(\V)$
inherits from $H^\sigma$ 
is the required group topology on $F_\sigma(\V)$.
\end{proof}

\begin{corollary}
\label{four:conditions:for:a:precompact:variety}
Let $\V$ be a precompact variety of groups. Under GCH, 
the following conditions are equivalent for every infinite cardinal $\sigma$:
\begin{itemize}
\item[(i)] the group $F_\sigma(\mathcal{V})$ admits a pseudocompact group topology;
\item[(ii)] the group $F_\sigma(\mathcal{V})$ admits a
strongly pseudocompact group topology;
\item[(iii)] the group $F_\sigma(\mathcal{V})$ admits a
\ssp\ group topology;
\item[(iv)] the group $F_\sigma(\mathcal{V})$ admits a
zero-dimensional \ssp\ group topology.
\end{itemize}
\end{corollary}
\begin{proof}
The implication (iv)$\Rightarrow$(iii) is trivial, while 
the
implications (iii)$\Rightarrow$(ii)$\Rightarrow$(i) follow from equation
\eqref{two:implications}.

(i)$\Rightarrow$(iv)
Let $\sigma$ be an arbitrary infinite cardinal. Assume (i); that  is,
$F_\sigma(\mathcal{V})$ admits a pseudocompact group topology.
Since $\sigma$ is infinite, $|F_\sigma(\mathcal{V})|=\sigma$ is an admissible cardinal by Definition \ref{def:admissible:cardinal}.
Since GCH holds, Lemma \ref{3.4DiS}~(iii) implies that $\sigma^\omega=\sigma$.
Applying Theorem 
\ref{precompact:variety:theorem}, we conclude that $F_\sigma(\mathcal{V})$ admits a zero-dimensional \ssp\ group topology.
\end{proof}

We use $\mathcal{G}$ 
for denoting the variety of all groups.

\begin{theorem}
\label{G:theorem}
For every infinite cardinal $\sigma$ such that $\sigma^{\omega}=\sigma$, the group $F_\sigma(\mathcal{G})$ admits a 
connected, locally connected,  \ssp\ group topology.
\end{theorem}
\begin{proof}
The group $H=SO(3,\R)$
of all rotations around the origin of the three-dimensional Euclidean space $\R^3$ under the operation of composition
is a 
compact metric group
satisfying $r_\mathcal{G}(\mathrm{SO}(3,\R))\geq\omega;$
see \cite{BM}.
Moreover, $H$
(trivially) belongs to the variety $\mathcal{G}$.
By Lemma 
\ref{indep_seq_pseudo}, 
there exists a dense $\mathcal G$-independent \ssp\ subset $Y$ of $H^{\sigma}$ 
such that $|Y|=\sigma.$
Since $H$ is both connected and locally connected,
applying Lemma \ref{lem:2.5} (with $X=Y$ and $\V=\mathcal{G}$), we conclude that 
$\langle Y\rangle\cong F_\sigma(\mathcal{G})$
is a connected, locally connected, \ssp\ (dense) subgroup of $H^\sigma$.
The subgroup topology that $\langle Y\rangle\cong F_\sigma(\mathcal{G})$
inherits from $H^\sigma$ 
is the required group topology on $F_\sigma(\mathcal{G})$.
\end{proof}

\begin{corollary}
\label{cor:3.9}
Under GCH, the following conditions are equivalent for every infinite cardinal $\sigma$:
\begin{itemize}
\item[(i)] 
$F_\sigma(\mathcal{G})$ admits a pseudocompact group topology;
\item[(ii)] 
$F_\sigma(\mathcal{G})$ admits a
strongly pseudocompact group topology;
\item[(iii)] 
$F_\sigma(\mathcal{G})$ admits a
\ssp\ group topology;
\item[(iv)] 
$F_\sigma(\mathcal{G})$ admits a
zero-dimensional \ssp\ group topology;
\item[(v)] 
$F_\sigma(\mathcal{G})$ admits a
connected, locally connected,  \ssp\ group topology.
\end{itemize}
\end{corollary}
\begin{proof}
By \cite[Lemma 5.3]{DiS}, the variety $\mathcal{G}$ is precompact. 
Hence, the equivalence of items (i)--(iv) follows from 
Corollary \ref{four:conditions:for:a:precompact:variety}.
The implication (i)$\Rightarrow$(v) was proved in Theorem
\ref{G:theorem}. 
The implication (v)$\Rightarrow$(iii) is trivial.
\end{proof}

We use 
$\mathcal{A}$ for denoting the variety of all Abelian groups.

\begin{theorem}\label{free-ssp}
Let $\V$ be a 
variety satisfying $\V\subseteq   \mathcal{A}$.
Then for every \sa\ cardinal $\tau\ge\omega$,
the $\V$-free group
$F_\tau(\V)$ admits a zero-dimensional \ssp\ group topology.
\end{theorem}
\begin{proof}
Since 
$\V\subseteq   \mathcal{A}$, the variety
$\V$ is precompact \cite{N}. By Definition \ref{def_variety}, there exists a zero-dimensional compact metric group $H\in\V$ with 
$r_\V(H)\geq\omega$. Let $\tau\ge\omega$ be a \sa\ cardinal.
By Definition 
\ref{def_sa}, there exists an infinite cardinal $\sigma$
satisfying $\sigma^\omega=\sigma\le\tau \le 2^\sigma$.
By Lemma \ref{v-free-group},
there exists a dense $\mathcal V$-independent \ssp\  subset $X$ of $H^{\sigma}$ such that $|X|=\tau$.
By
Lemma \ref{lem:2.5}, 
$\langle X\rangle\cong F_\tau(\V)$
is a zero-dimensional \ssp\ (dense) subgroup of $H^\sigma$.
The subgroup topology that $\langle X\rangle\cong F_\tau(\V)$
inherits from $H^\sigma$ 
is the required group topology on $F_\tau(\V)$.
\end{proof}

\begin{corollary}
\label{corollary:for:abelian:varieties}
Let $\V$ be a 
variety satisfying $\V\subseteq   \mathcal{A}$.
Under SCH, the following conditions
are equivalent for every infinite cardinal $\tau$:
\begin{enumerate}
\item[(i)] 
$F_\tau(\V)$ admits a pseudocompact group topology;
\item[(ii)] 
$F_\tau(\V)$ admits a strongly pseudocompact group topology;
\item[(iii)] 
$F_\tau(\V)$ admits a \ssp \ group topology;
\item[(iv)]
$F_\tau(\V)$ admits a zero-dimensional \ssp\ group topology. 
\end{enumerate} 
\end{corollary}
\begin{proof}
(i)$\Rightarrow$(iv) 
By (i), $F_\tau(\V)$ admits a pseudocompact group topology.
Since $\tau$ is infinite, $|F_\tau(\V)|=\tau$ is an admissible cardinal by Definition \ref{def:admissible:cardinal}.
Since SCH holds, $\tau$ is \sa\ by Lemma \ref{3.4DiS}~(ii).
Applying Theorem \ref{free-ssp}, we conclude that 
$F_\tau(\V)$ admits a zero-dimensional \ssp\ group topology.

The implication (iv)$\Rightarrow$(iii) is trivial. The implications (iii)$\Rightarrow$(ii)$\Rightarrow$(i)
follow from equation \eqref{two:implications}.
\end{proof}

\begin{theorem}\label{free-abelian-ssp}
For every \sa\ cardinal $\tau\ge\omega$,
the $\mathcal{A}$-free group
$F_\tau(\mathcal{A})$ 
admits a connected, locally connected, \ssp\ group topology.
\end{theorem}
\begin{proof}
Let $\tau\ge\omega$ be a \sa\ cardinal.
By Definition 
\ref{def_sa}, there exists an infinite cardinal $\sigma$
satisfying $\sigma^\omega=\sigma\le\tau \le 2^\sigma$.
Since $r_{\mathcal{A}}(\T)=\C\ge\omega$,
by Lemma \ref{v-free-group},
there exists a dense $\mathcal A$-independent \ssp\  subset $X$ of $\T^{\sigma}$ such that $|X|=\tau$.
Since $\T$ is both connected and locally connected, by
Lemma \ref{lem:2.5}, 
$\langle X\rangle\cong F_\tau(\mathcal{A})$
is a connected, locally connected, \ssp\ (dense) subgroup of $\T^\sigma$.
The subgroup topology that $\langle X\rangle\cong F_\tau(\mathcal{A})$
inherits from $\T^\sigma$ 
is the required group topology on $F_\tau(\mathcal{A})$.
\end{proof}

\begin{corollary}
\label{cor:3.13}
Under SCH, the following conditions are equivalent for every infinite cardinal $\tau$:
\begin{itemize}
\item[(i)] $F_\tau(\mathcal{A})$ admits a pseudocompact group topology;
\item[(ii)] $F_\tau(\mathcal{A})$ admits a
strongly pseudocompact group topology;
\item[(iii)] $F_\tau(\mathcal{A})$ admits a
\ssp\ group topology;
\item[(iv)] $F_\tau(\mathcal{A})$ admits a
zero-dimensional \ssp\ group topology;
\item[(v)] $F_\tau(\mathcal{A})$ admits a
connected, locally connected,  \ssp\ group topology.
\end{itemize}
\end{corollary}
\begin{proof}
The equivalence of items (i)--(iv) follows from 
Corollary \ref{corollary:for:abelian:varieties}.
The implication (i)$\Rightarrow$(v) was proved in Theorem
\ref{free-abelian-ssp}. 
The implication (v)$\Rightarrow$(iii) is trivial.
\end{proof}

\section{The algebraic structure of pseudocompact torsion Abelian groups}

The main result in this section is Theorem \ref{pseudocompact:torsion:group:are:sums:of:nice:groups}
which describes the algebraic structure of pseudocompact torsion Abelian groups. 
One can deduce this theorem from \cite[Theorem 6.2]{DiS} but we decided to give a separate proof of it.
Another description of the algebraic structure of
pseudocompact torsion Abelian groups can be found in 
\cite[Theorem 3.19]{CR} and \cite[Theorem 6.2]{DiS}.

\begin{definition}
\label{def:good:sequence}
For $m,n\in\N^+$ with $m\le n$,
a finite sequence $(\tau_m,\tau_{m+1},\dots,\tau_n)$ of cardinals will be called:
\begin{itemize}
\item[(i)]
 {\em good\/} provided that, for every integer $k$ such that $m\le k\le n$, the cardinal
$\max_{k\le i\le n} \tau_i$ 
 is either finite or admissible;
\item[(ii)] {\em very good\/}
if $\tau_n$ is either finite or admissible and 
$\tau_i< \tau_n$ for every $i\in\N^+$ with $m\le i< n$.
\end{itemize}
\end{definition}

\begin{remark}
\label{good:remark}
Assume that  $m,l,n\in\N^+$ and $m\le l\le n$. If $(\tau_m,\tau_{m+1},\dots,\tau_n)$ is a good sequence of cardinals,
then the sequence
$(\tau_{l},\tau_{l+1},\dots,\tau_n)$
is good as well.
\end{remark}

The next lemma shows that each good sequence of cardinals can be partitioned into finitely many adjacent very good subsequences.

\begin{lemma}
\label{partition:lemma}
Assume that $m,n\in\N^+$, $m\le n$ and $(\tau_m,\tau_{m+1},\dots,\tau_n)$ is a good sequence of cardinals.
Then there exists a strictly increasing sequence 
$(s_{0}, s_{1},\dots,s_k)$ of integers such that 
$s_0= m-1$, $s_k=n$ and the sequence
$(\tau_{s_i+1},\tau_{s_i+2},\dots, \tau_{s_{i+1}})$
is very good for every $i=0,\dots,k-1$.
\end{lemma}
\begin{proof}
We shall prove this lemma by induction on $n-m$.

\smallskip
{\em Basis of induction\/}. Suppose that
the sequence $(\tau_n)$ is good. 
By item (i) of Definition \ref{def:good:sequence},
$\tau_n$ is either finite or admissible.
It follows from item (ii) of Definition \ref{def:good:sequence} that
the sequence $(\tau_n)$ is very good. 

\smallskip
{\em Inductive step\/}. Suppose that $j\in\N^+$ and our lemma has been already proved for all good sequences $(\tau_m,\tau_{m+1},\dots,\tau_n)$ such that $n-m<j$.
Fix a good sequence $(\tau_m,\tau_{m+1},\dots,\tau_n)$ with
$n-m=j$.

The set $K=\{k\in\N^+: m\le k\le n, \tau_k=\max_{m\le i\le n} \tau_i\}$ is non-empty, so we can define
$l=\min K$.

We claim that the sequence $(\tau_m,\dots,\tau_l)$ is very good. Since the sequence 
$(\tau_m,\tau_{m+1},\dots,\tau_n)$ is good by our assumption, the cardinal $\max_{m\le i\le n} \tau_i$ is either finite or admissible by Definition
\ref{def:good:sequence}~(i). Since $l\in K$, this means that 
$\tau_l=\max_{m\le i\le n} \tau_i$ is either finite or admissible.
Since $l=\min K$, it follows from the definition of $K$ that
$\tau_i<\tau_l$ whenever $m\le i< l$.
Therefore, the sequence $(\tau_m,\dots,\tau_l)$ is very good
by Definition \ref{def:good:sequence}~(ii).
Now we consider two cases.

\smallskip
{\em Case 1.\/} $l=n$. In this case,
the original sequence
$(\tau_m,\tau_{m+1},\dots,\tau_n)$ is very good.

\smallskip
{\em Case 2.\/} $l<n$. 
  By Remark \ref{good:remark},
 the sequence $(\tau_{l+1},\tau_{l+2},\dots,\tau_n)$ is good.
Since $n-(l+1)< n-l\le n-m=j$,
we can apply our inductive assumption to this sequence to get
 a strictly increasing sequence 
$(s_{1}, s_{2},\dots,s_k)$ of integers such that 
$s_1= l$, $s_k=n$ and the sequence
$(\tau_{s_i+1},\tau_{s_i+2},\dots, \tau_{s_{i+1}})$
is very good for every $i=1,\dots,k-1$.
Let $s_0=m-1$.
Now
$(s_{0}, s_{1},\dots,s_k)$ is the desired sequence.
\end{proof}

Good sequences of cardinals typically appear in the situation described by our next lemma. 
Even though this lemma can be derived from 
\cite[Theorem 6.2]{DiS}, 
we include its simple proof
for convenience of the reader. 

\begin{lemma}
\label{good:sequences:from:pseudocompact:groups}
If $G$
is a pseudocompact group
such that $G\cong \bigoplus_{i=m}^n \Z(p^i)^{(\tau_i)}$
for some prime number $p$,
positive integers
$m,n\in\N^+$ and cardinals $\tau_m,\tau_{m+1},\dots,\tau_n$,
then the sequence 
$(\tau_m,\tau_{m+1},\dots,\tau_n)$ is good.
\end{lemma}
\begin{proof}
It suffices to show that,  for every integer $k$ such that 
$m\le k\le n$, the cardinal
$\mu_k=\max_{k\le i\le n} \tau_i$
is either finite or admissible.
Fix an integer $k$ satisfying $m\le k\le n$.
If all cardinals $\tau_k,\tau_{k+1},\dots,\tau_n$ are finite,
then $\mu_k$ is finite. Assume now that at least one of the cardinals $\tau_k,\tau_{k+1},\dots,\tau_n$ is infinite.
Then
the group
\begin{equation}
p^{k-1} G=
\bigoplus_{i=k}^n \Z(p^{i-k+1})^{(\tau_i)}
\end{equation}
is infinite, so 
$|p^{k-1} G|=
\max_{k\le i\le n} \tau_i=\mu_k$.
Observe that $p^{k-1} G$ is a continuous image of the pseudocompact group $G$ 
under the map $g\mapsto p^{k-1} g$, so
$p^{k-1} G$ is pseudocompact as well.
Since 
the cardinality
$\mu_k$ of the pseudocompact group $p^{k-1} G$ is infinite, it must be admissible by Definition \ref{def:admissible:cardinal}.
\end{proof}

\begin{definition}
\label{def:nice:group}
We shall say that a group $G$ is {\em nice\/} provided that 
there exist 
a prime number $p$,
positive integers $m,n\in\N^+$
and a very good sequence  $(\tau_m,\tau_{m+1},\dots,\tau_n)$ of cardinals 
such that $G\cong \bigoplus_{i=m}^n \Z(p^i)^{(\tau_i)}$.
\end{definition}

Let $p$ be a prime number.
Recall that
an Abelian group $G$ is called a {\em $p$-group\/} provided that for every $g\in G$ one 
can find $k\in\N$ such that $p^k g=0$.

\begin{lemma}
\label{decomposition:of:pseudocompact:p-groups}
A non-trivial
pseudocompact bounded torsion Abelian $p$-group 
can be decomposed into a finite direct sum of nice groups.
\end{lemma}
\begin{proof}
Let $G$ be a non-trivial pseudocompact bounded torsion $p$-group. Then
\begin{equation}
\label{eq:Prufer:dec}
G\cong \bigoplus_{j=1}^n \Z(p^j)^{(\tau_j)}
\end{equation}
 for a suitable $n\in\N^+$ and cardinals $\tau_1,\tau_2,\dots,\tau_n$; see \cite[Theorem 17.2]{F}. 
By Lemma
 \ref{good:sequences:from:pseudocompact:groups},
 the sequence 
 $(\tau_1,\tau_2,\dots,\tau_n)$ is good. 
Applying Lemma
 \ref{partition:lemma} to the sequence $(\tau_1,\tau_2,\dots,\tau_n)$,
we obtain 
a strictly increasing sequence 
$(s_{0}, s_{1},\dots,s_k)$ of integers such that 
$s_0= 0$, $s_k=n$ and the sequence
$(\tau_{s_i+1},\tau_{s_i+2},\dots, \tau_{s_{i+1}})$
is very good for every $i=0,\dots,k-1$.
By Definition \ref{def:nice:group},
the group 
$$
G_i=
\bigoplus_{j=s_i+1}^{s_{i+1}} \Z(p^j)^{(\tau_j)}
$$
is nice 
for every $i=0,\dots,k-1$.
Clearly,
$G=\bigoplus_{i=0}^{k-1} G_i$.
\end{proof}

\begin{theorem}
\label{pseudocompact:torsion:group:are:sums:of:nice:groups}
Every non-trivial pseudocompact torsion Abelian group is isomorphic to a finite direct sum of nice groups. 
\end{theorem}
\begin{proof}
Let $G$ be a non-trivial pseudocompact torsion Abelian group.
Then $G$ is bounded torsion \cite[Corollary 3.9]{DiS},
and so there exists a finite set 
$P$ of prime numbers
such that
$G=\bigoplus_{p\in P} G_p$, where $G_p$ is a non-trivial (bounded torsion) $p$-group
for every $p\in P;$ see \cite[Theorem 17.2]{F}. 

Let $p\in P$ be arbitrary. For every $q\in P\setminus \{p\}$,
$G_q$ has order $q^{k_q}$ for some $k_q\in \N^+$; that is,
$q^{k_q} G_q=\{0\}$.
Define $n=\prod_{q\in P\setminus\{p\}} q^{k_q}$.
Since $G_p$ is $q$-divisible for every prime number 
$q$ different from $p$,
it follows that $nG=nG_p=G_p$.
Since $nG$ is an image of $G$ under the continuous map $g\mapsto nG$ and $G$ is pseudocompact, we conclude that
$G_p$ is also pseudocompact.
Now
we can 
apply Lemma \ref{decomposition:of:pseudocompact:p-groups} to $G_p$ to get a decomposition of $G_p$ into a finite direct sum of nice groups.
\end{proof}

\section{Selectively sequentially pseudocompact topologies on torsion Abelian groups}

\begin{lemma}
\label{claim:1}
Assume that 
$p$ is a prime number,
$m,n\in\N^+$, $\tau_m,\tau_{m+1},\dots,\tau_n$ are cardinals such that  $\tau_i\leq\tau_n$ for all $i=m,\ldots,n$, and  
$\tau_n$ is \sa.
Then
$G=\bigoplus_{i=m}^n \Z(p^i)^{(\tau_i)}$ can be equipped with a \ssp\ group topology.
\end{lemma}
\begin{proof}
Suppose that $\tau_n$ is finite.
Since $\tau_i\le\tau_j$ for all $i=m,\ldots,n$,
the group $G$ is finite. The discrete topology on $G$ is compact, 
so also \ssp.

Suppose now that $\tau_n$ is infinite.
Since $\tau_n$ is \sa, Definition \ref{def_sa}
implies that $\sigma^{\omega}=\sigma\leq\tau_n\leq 2^{\sigma}$ 
for some cardinal $\sigma$.

Let $H=\Z(p^n)^{\omega}$.
By Lemma \ref{v-free-group} applied to this $H$,  the variety $\mathcal{A}_{ p^n}=\{G: G$ is an Abelian group such that $p^ng=0$ for all $g\in G\}$ taken as 
$\mathcal{V}$ and the cardinal $\tau_n$ taken as $\tau$, we can find
a \ssp\ dense $\mathcal{A}_{p^n} $-independent subset $X$ of $H^{\sigma}$ and a set $D\subseteq \sigma$
such that $|X|=\tau_n, |D|=|\sigma\setminus D|=\sigma$ and 
$\langle X\rangle\cap H^{D}=\{0\}$, where 
$H^{D}=\{f\in H^{\sigma}:f(\alpha)=0$ for all $\alpha\in \sigma\setminus D\}$.
Since $X$ is $\mathcal{A}_{p^n} $-independent and $|X|=\tau_n$,
we have 
$\langle X\rangle\cong \Z(p^n)^{(\tau_n)}.$ Let $\varphi: \Z(p^n)^{(\tau_n)}\to
\langle X\rangle$ be an isomorphism. 
Define
\begin{equation}
\label{eq:G'}
G'=
\left\{\begin{array}{ll}
\{0\}  &
\textrm{ if } n=m\\
\bigoplus_{i=m}^{n-1} \Z(p^{i})^{(\tau_{i})} & \textrm{ if } n>m.
\\
\end{array}
\right.
\end{equation}
Clearly, $G=G'\oplus \Z(p^n)^{(\tau_n)}$.

We claim that there is a monomorphism $\psi:G'\to H^D$.
This is clear for $n=m$, as $G'=\{0\}$ by \eqref{eq:G'}.
Suppose now that $n>m$.
Then $G'=\bigoplus_{i=m}^{n-1} \Z(p^{i})^{(\tau_{i})}$ by \eqref{eq:G'}.
Let $D=\bigcup_{i=m}^{n-1} D_i$ be a decomposition of $D$ into pairwise disjoint sets $D_i$ such that $|D_i|=|D|$ for all $i=m,\dots,n-1$.

Let $i=m,\dots,n-1$ be arbitrary.
Note that $H^{D_i}\cong H^{D}\cong \Z(p^n)^{D}$,
so $r_{\mathcal{A}_{p^n}}(H^{D_i})\ge 2^{|D|}=2^{\sigma}$
by \cite[Lemma 4.1]{DiS}.  
Since $\tau_i\le\tau_n\le 2^{\sigma}\leq r_{\mathcal{A}_{p^n}}(H^{D_i})$,
we can fix a monomorphism 
$\psi_i\Z(p^n)^{(\tau_i)}\to H^{D_i}$.

Let $\psi:G'=\bigoplus_{i=m}^{n-1} \Z(p^{i})^{(\tau_{i})}\to \bigoplus_{i=m}^{n-1} H^{D_i}\cong H^D$
be the unique monomorphism extending each of $\psi_i$.
Since $\psi(G')\subseteq H^D$, $\varphi(\Z(p^n)^{(\tau_n)})=
\langle X\rangle$
and
$H^D\cap \langle X\rangle=\{0\}$, 
there is a unique monomorphism
$\chi: G=G'\oplus \Z(p^n)^{(\tau_n)}\to H^D\oplus \langle X\rangle\subseteq H^\sigma$ extending both $\varphi$ and $\psi$.

Since 
$\langle X\rangle$ is a dense \ssp\ subgroup of $H^\sigma$ and 
$\langle X\rangle =\varphi(\Z(p^n)^{(\tau_n)})\subseteq 
\chi(G)$,
the subgroup $\langle X\rangle$ of $\chi(G)$ is dense 
in $\chi(G)$.
Since
$\langle X\rangle$ is \ssp, 
Proposition \ref{dense:ssp} implies that
$\chi(G)$ is 
\ssp\ as well.
Since $\chi:G\to \chi(G)$ is an isomorphism,
$G$ admits a \ssp\ group topology.
\end{proof}

\begin{theorem}
\label{thm:for:torsion:groups}
Under SCH,  the following conditions are equivalent for every
torsion Abelian group $G$:
\begin{itemize}
\item[(i)] $G$ has a pseudocompact group topology;
\item[(ii)] $G$ has a strongly pseudocompact group topology;
\item[(iii)] $G$ has a \ssp\ group topology.
\end{itemize} 
\end{theorem}

\proof
The implications (iii)$\Rightarrow $(ii)$\Rightarrow$(i) follow from equation \eqref{two:implications}.

(i)$\Rightarrow $(iii). 
Without loss of generality, we shall assume that $G$ is non-trivial.
By Theorem \ref{pseudocompact:torsion:group:are:sums:of:nice:groups},
$G\cong \bigoplus_{k=1}^j G_k$, where $j\in\N^+$ and each 
$G_k$ is a nice group. 
By Proposition \ref{product:ssp},
it suffices to prove that each $G_k$ admits a \ssp\ group topology.

Fix $k=1,\dots,j$. Since $G_k$ is nice,
we can apply Definition \ref{def:nice:group}
to fix a prime number $p$,
positive integers
$m,n\in\N^+$ and a very good sequence  $(\tau_m,\tau_{m+1},\dots,\tau_n)$ of cardinals 
such that $G_k\cong \bigoplus_{i=m}^n \Z(p^i)^{(\tau_i)}$.
By Definition \ref{def:good:sequence}~(ii),
$\tau_n$ is either finite or admissible and 
$\tau_i< \tau_n$ for every $i\in\N^+$ with $m\le i<  n$.
If $\tau_n$ is finite, then 
it is selectively admissible by Definition \ref{def_sa}.
Suppose now that $\tau_n$ is 
admissible.
Since SCH holds, $\tau_n$ is \sa\ by Lemma \ref{3.4DiS}~(ii).
By Lemma \ref{claim:1},
$G_k$ admits a \ssp\ group topology.
\endproof

\section{Selectively sequentially pseudocompact topologies on torsion-free Abelian groups}

Recall that a subgroup $H$ of an Abelian group $G$ is said to be {\em essential\/} in $G$ provided that $N\cap H\not=\{0\}$ for every non-zero subgroup $N$ of $G$.  
\begin{lemma}
\label{essential:lemma}
If $F$ is a maximal $\mathcal{A}$-free subgroup of an Abelian group $G$,
then
$H=t(G)+F=t(G)\oplus F$ is an essential subgroup of $G$.
\end{lemma}
\begin{proof}
Since $F$ is torsion-free and $t(G)$ is torsion, $t(G)\cap F=\{0\}$, which means that the sum $H=t(G)+F$ is direct.

Let $N$ be a non-zero subgroup of $G$. Fix a non-zero element $x\in N$. If $x$ is torsion, then $x\in t(G)$, so $x\in t(G)\cap N\subseteq H\cap N$, which implies $H\cap N\not=\{0\}$. Suppose now that $x$ is non-torsion. Then $\langle x\rangle$ is an infinite cyclic subgroup of $G$.
If $F\cap \langle x\rangle=\{0\}$, then $F'=F+\langle x\rangle=
F\oplus \langle x\rangle$ would be an $\mathcal{A}$-free subgroup of $G$ containing $F$ as a proper subgroup, contradicting the maximality of $F$. Therefore,
$F\cap \langle x\rangle\not=\{0\}$, which implies $H\cap N\not=\{0\}$. 
\end{proof}

\begin{theorem}\label{abelian}
Suppose that $G$ is an Abelian group and $\sigma^{\omega}=\sigma\leq r_\mathcal{A}(G)\leq |G|\leq 2^{\sigma}$
for some infinite cardinal $\sigma.$ Then $G$ is isomorphic to a dense connected, locally connected, \ssp\ subgroup 
of $\T^{\sigma}.$ 
\end{theorem}

\proof
Let $F$ be a maximal $\mathcal{A}$-free subgroup of the group $G$; the existence of such $F$ follows from Zorn's lemma.
Then $r_\mathcal{A}(F)=r_\mathcal{A}(G)\ge\sigma^\omega\ge\C$, as $\sigma$ is infinite.
Since $r_{\mathcal{A}}(F)=r_{\mathcal{A}}(G)\le |G|\le 2^\sigma$,
the cardinal $r_{\mathcal{A}}(F)$ is \sa\ by Definition \ref{def_sa}.
Arguing as in the proof of Theorem
\ref{free-abelian-ssp}, we can fix a monomorphism
$f: F\to \T^\sigma$ such that $f(F)$ is a dense \ssp\ subgroup
of $\T^\sigma$.

Since $\T^\sigma$ is  divisible, $r_\mathcal{A}(t(G))=0\le 2^\sigma=r_\mathcal{A}(\T^\sigma)$ and
$r_{\mathcal{A}_p}(t(G))\le |G|\le 2^\sigma=r_{\mathcal{A}_p}(\T^\sigma)$ for every prime number $p$,
we can use \cite[Lemma 3.16]{DS-Forcing}
to find a monomorphism $g: t(G)\to\T^\sigma$.

Let $h:H=t(G)\oplus F\to \T^\sigma$ be the unique homomorphism
extending both $g$ and $f$.
Since $h(t(G))$ is torsion and $h(F)=f(G)$ is torsion-free,
$h(t(G))\cap h(F)=\{0\}$.
Since both $g$ and $f$ are monomorphisms, so is $h$; see
\cite[Lemma 3.10~(ii)]{DS-Forcing}.
Since $\T^\sigma$ is divisible, there exists a homomorphism
$\varphi:G\to\T^\sigma$ extending $h$; see
\cite[Theorem 21.1]{F}.
Since $H$ is essential in $G$ by Lemma \ref{essential:lemma}
and $h$ is a monomorphism,  $\varphi$ is a monomorphism as well; see
\cite[Lemma 3.5]{DS-Forcing}.
Therefore, $G$ and $\varphi(G)$ are isomorphic.

Since $\varphi(F)=f(F)$ is a dense \ssp\ subgroup of $\T^\sigma$
and $\varphi(F)\subseteq \varphi(G)$, the group $\varphi(G)$ is 
\ssp\ by Proposition \ref{dense:ssp}.
Finally, 
since $\T$ is connected and locally connected, \cite[Fact 2.11]{DiS} implies that $\varphi(G)$ is also connected and locally connected.
\endproof

\begin{theorem}
\label{sch_connected}
\label{rank:coincides:with:cardinaity}
Under SCH, the following conditions are equivalent for every 
Abelian group $G$ satisfying $r_\mathcal{A}(G)=|G|$:
\begin{itemize}
\item[(i)] $G$ admits a pseudocompact group topology;
\item[(ii)] $G$ admits a strongly pseudocompact group topology;
\item[(iii)] $G$ admits a \ssp\ group topology;
\item[(iv)] $G$ admits a connected, locally connected, \ssp\ group topology.
\end{itemize}
\end{theorem}

\begin{proof} 
The implication (iv)$\Rightarrow$(iii) is trivial, while 
the
implications (iii)$\Rightarrow$(ii)$\Rightarrow$(i) follow from equation
\eqref{two:implications}.

(i)$\Rightarrow$(iv)
Suppose that $G$ admits a pseudocompact group topology. 
Observe that the equality $r_\mathcal{A}(G)= |G|$ implies that $G$ is infinite, so
$|G|$ is admissible by Definition \ref{def:admissible:cardinal}.
Since SCH is assumed, $|G|$ is \sa\
by Lemma \ref{3.4DiS}~(ii).
By Definition \ref{def_sa},
there exists an infinite cardinal $\sigma$ such that 
$\sigma^{\omega}=\sigma\leq  |G|\leq 2^{\sigma}.$
Since
$r_\mathcal{A}(G)= |G|$, 
the assumptions of 
Theorem \ref{abelian} are satisfied.
Applying this theorem, we conclude that
$G$ is isomorphic to a dense 
connected, locally connected, \ssp\ subgroup 
of $\T^{\sigma}.$ 
\end{proof}

\begin{corollary}\label{sch_torsion-free}
Under SCH, the following conditions are equivalent for every 
torsion-free
Abelian group $G$:
\begin{itemize}
\item[(i)] $G$ admits a pseudocompact group topology;
\item[(ii)] $G$ admits a strongly pseudocompact group topology;
\item[(iii)] $G$ admits a \ssp\ group topology;
\item[(iv)] $G$ admits a connected, locally connected, \ssp\ group topology.
\end{itemize}
\end{corollary}

\begin{proof}
Since $G$ is torsion-free,
$G$ must be infinite.

(i)$\Rightarrow$(iv)
Suppose that $G$ admits a pseudocompact group topology. 
Since $G$ is infinite,
$|G|\geq r_{\mathcal{A}}(G)\ge\C$ by \cite[Theorem 3.8]{DiS}.
Since $G$ is an uncountable, torsion-free Abelian group, $r_\mathcal{A}(G)= |G|$.
By the implication (i)$\Rightarrow$(iv) of Theorem \ref{rank:coincides:with:cardinaity}, $G$ admits a connected, locally connected, \ssp\ group topology.

The implication (iv)$\Rightarrow$(iii) is trivial, while 
the
implications (iii)$\Rightarrow$(ii)$\Rightarrow$(i) follow from equation
\eqref{two:implications}.
\end{proof}

\begin{corollary}\label{sch_non-torsion}
Under SCH, the following conditions are equivalent for every 
Abelian group $G$:
\begin{itemize}
\item[(i)] $G/t(G)$ admits a pseudocompact group topology;
\item[(ii)] $G/t(G)$ admits a strongly pseudocompact group topology;
\item[(iii)] $G/t(G)$ admits a \ssp\ group topology;
\item[(iv)] $G/t(G)$ admits a connected, locally connected, \ssp\ group topology.
\end{itemize}
\end{corollary}

\begin{proof}
If $G$ is torsion, then $G/t(G)$ is the trivial group, and so all four conditions are equivalent for $G/t(G)$.
If $G$ is non-torsion, then $G/t(G)$ is a torsion-free group, so 
Corollary \ref{sch_torsion-free} applies to $G/t(G)$.
\end{proof}

\begin{theorem}
\label{thm:6.5}
Under GCH, the following conditions are equivalent for every 
Abelian group $G$ satisfying the inequality
$|G|\leq 2^{r_\mathcal{A}(G)}$:
\begin{itemize}
\item[(i)] $G$ admits a pseudocompact group topology;
\item[(ii)] $G$ admits a strongly pseudocompact group topology;
\item[(iii)] $G$ admits a \ssp\ group topology;
\item[(iv)] $G$ admits a connected, locally connected, \ssp\ group topology.
\end{itemize}
\end{theorem}

\begin{proof} 
If $r_{\mathcal{A}}(G)=0$, then $|G|\le 2^{r_\mathcal{A}(G)}=1$ by our assumption, so $G$ is trivial, and all items (i)--(iv) are equivalent for the trivial group. Therefore, from now on we shall assume that $r_{\mathcal{A}}(G)\ge 1$; in particular, $G$ is  infinite and non-torsion.

(i)$\Rightarrow$(iv)
Suppose that $G$ admits a pseudocompact group topology. 
Since $G$ is non-torsion,
$\sigma=r_{\mathcal{A}}(G)$ is admissible by \cite[Corollary 1.19]{DGB}.
Since GCH is assumed, 
$\sigma^\omega=\sigma$ by Lemma \ref{3.4DiS}~(iii).
Since
$|G|\leq 2^\sigma$,
the assumptions of 
Theorem \ref{abelian} are satisfied.
Applying this theorem, we conclude that
$G$ is isomorphic to a dense 
connected, locally connected, \ssp\ subgroup 
of 
$\T^\sigma$.

The implication (iv)$\Rightarrow$(iii) is trivial, while 
the
implications (iii)$\Rightarrow$(ii)$\Rightarrow$(i) follow from equation~\eqref{two:implications}.
\end{proof}

\section{Two open questions}

We finish with two 
concrete versions 
of general Question \ref{main:problem} related to our results.

\begin{question}
Can Corollaries \ref{four:conditions:for:a:precompact:variety} and \ref{cor:3.9} be proved in ZFC? Or at least, can GCH be weakened to SCH in the assumption of these corollaries?
\end{question}

\begin{question}
Can Theorems \ref{thm:for:torsion:groups}, \ref{sch_connected}, \ref{thm:6.5}
and Corollaries \ref{corollary:for:abelian:varieties}, \ref{cor:3.13}, \ref{sch_torsion-free}, \ref{sch_non-torsion}
be proved in ZFC? 
\end{question}

\medskip
\noindent
{\bf Acknowledgement:\/} 
This paper was written during the first listed author's stay at the Department of Mathematics 
of Faculty of Science of Ehime University (Matsuyama, Japan)
in the capacity of Visiting Foreign Researcher 
with 
the financial support from
CONACyT of M\'exico: Estancias Posdoctorales al Extranjero propuesta No. 263464. The first author
would like to thank CONACyT for its support and the host institution for its hospitality.

\end{document}